\documentclass[12pt, leqno]{article}
\usepackage{amsfonts,cite}
\usepackage{amsmath, amssymb}
\usepackage{euler}
\usepackage{ marvosym }
\usepackage{enumitem}
\usepackage[amsmath,thmmarks]{ntheorem}
\usepackage{chngcntr,color}

\counterwithin*{equation}{section}
\counterwithin*{equation}{subsection}

\newtheorem{theorem}{Theorem}

\newtheorem{lemma}{Lemma}
\newtheorem*{remark}{Remark}

\newenvironment{proof}[1][Proof]{\noindent\textbf{#1.} }{\ \rule{0.5em}{0.5em}}

\input{tcilatex}

\begin{document}

\title{Interpolation of data by smooth non-negative functions}
\date{\today }
\author{Charles Fefferman, Arie Israel, Garving K. Luli \thanks{%
The first author is supported in part by NSF grant DMS-1265524, AFOSR
grant FA9550-12-1-0425, and Grant No 2014055 from the United States-Israel Binational Science Foundation (BSF). The third author is supported in part by NSF grant
DMS-1355968 and a start-up fund from UC Davis.}}
\maketitle

\section*{Introduction}
Continuing from \cite{fil-2016}, we prove a finiteness principle for interpolation of data by nonnegative $C^m$ functions. Our result raises the hope that one can start to understand constrained interpolation problems in which e.g. the interpolating function $F$ is required to be nonnegative.

Let us recall some notation used in \cite{fil-2016}. 

We fix positive integers $m$, $n$. We write $C^{m}\left( \mathbb{R}^{n}%
\right) $ to denote the Banach space of all real valued locally $C^{m}$ functions $F$ on $\mathbb{R}^{n}$, for which the
norm 
\begin{equation*}
\left\Vert F\right\Vert _{C^{m}\left( \mathbb{R}^{n}\right)
}:=\sup_{x\in \mathbb{R}^{n}}\max_{\left\vert \alpha \right\vert \leq
m}\left\vert \partial ^{\alpha }F\left( x\right) \right\vert
\end{equation*}%
is finite. 

We will also work with the function space $C^{m-1,1}(\mathbb{R}^n)$. A given continuous function $F:\mathbb{R}^{n}\rightarrow 
\mathbb{R}$ belongs to $C^{m-1,1}\left( \mathbb{R}^{n}\right) $ if and only
if its distribution derivatives $\partial ^{\beta }F$ belong to $L^{\infty
}\left( \mathbb{R}^{n}\right) $ for $\left\vert \beta \right\vert \leq m$.
We may take the norm on $C^{m-1,1}\left( \mathbb{R}^{n}\right) $ to be 
\begin{equation*}
\left\Vert F\right\Vert _{C^{m-1,1}\left( \mathbb{R}^{n}\right)
}=\max_{\left\vert \beta \right\vert \leq m}\text{ess.}\sup_{x\in \mathbb{R}%
^{n}}\left\vert \partial ^{\beta }F\left( x\right) \right\vert \text{.}
\end{equation*}%

Expressions $c\left( m,n\right) $, $C\left( m,n\right) $, $k\left(
m,n\right) $, etc. denote constants depending only on $m$, $n$; these
expressions may denote different constants in different occurrences. Similar
conventions apply to constants denoted by $C\left( m,n,D\right)$, $k\left(
D\right)$, etc.

If $X$ is any finite set, then $\#\left( X\right) $ denotes the number of
elements in $X$.

We are now ready to state our main theorem. 

\begin{theorem}
\label{Th4}For large enough $k^{\#}=k\left( m,n\right) $ and $C^{\#}=C\left(
m,n\right) $ the following hold.

\begin{description}
\item[(A) $C^{m}$ FLAVOR] Let $f:E\rightarrow [0,\infty)$ with $%
E\subset \mathbb{R}^{n}$ finite. Suppose that for each $S\subset E$ with $%
\#\left( S\right) \leq k^{\#}$, there exists $F^{S}\in C^{m}\left( \mathbb{R}%
^{n}\right) $ with norm $\left\Vert F^{S}\right\Vert _{C^{m}\left( \mathbb{R}%
^{n}\right) }\leq 1$, such that $F^{S}=f\ $on $S$ and $F^{S}\geq 0$ on $%
\mathbb{R}^{n}$. \newline
Then there exists $F\in C^{m}\left( \mathbb{R}^{n}\right) $ with norm $%
\left\Vert F\right\Vert _{C^{m}\left( \mathbb{R}^{n}\right) }\leq C^{\#}$,
such that $F=f$ on $E$ and $F\geq 0$ on $\mathbb{R}^{n}$.

\item[(B) $C^{m-1,1}$ FLAVOR] Let $f:E\rightarrow [0,\infty)$ with $%
E\subset \mathbb{R}^{n}$ arbitrary. Suppose that for each $S\subset E$ with $%
\#\left( S\right) \leq k^{\#}$, there exists $F^{S}\in C^{m-1,1}\left( 
\mathbb{R}^{n}\right) $ with norm $\left\Vert F^{S}\right\Vert
_{C^{m-1,1}\left( \mathbb{R}^{n}\right) }\leq 1$, such that $F^{S}=f\ $on $S$
and $F^{S}\geq 0$ on $\mathbb{R}^{n}$. \newline
Then there exists $F\in C^{m-1,1}\left( \mathbb{R}^{n}\right) $ with norm $%
\left\Vert F\right\Vert _{C^{m-1,1}\left( \mathbb{R}^{n}\right) }\leq C^{\#}$%
, such that $F=f$ on $E$ and $F\geq 0$ on $\mathbb{R}^{n}$.
\end{description}
\end{theorem}

Our interest in Theorem \ref{Th4} arises in part from its possible connection to the interpolation algorithm of Fefferman-Klartag \cite{fb1,fb2}. Given a function $f:E\rightarrow \mathbb{R}$ with $E \subset \mathbb{R}^n$ finite, the goal of \cite{fb1,fb2} is to compute a function $F \in C^m(\mathbb{R}^n)$ such that $F=f$ on $E$, with $||F||_{C^m(\mathbb{R}^n)}$ as small as possible up to a  factor $C(m,n)$. Roughly speaking, the algorithm in \cite{fb1,fb2} computes such an $F$ using $O(N \log N)$ computer operations, where $N=\#(E)$. The algorithm is based on an easier version \cite{f-2005} of Theorem \ref{Th4}.  Our present result differs from the easier version in that we have added the hypothesis $F^S \geq 0$ and the conclusion $F\geq 0$. Accordingly, Theorem \ref{Th4} raises the hope that we can start to understand constrained interpolation problems, in which e.g. the interpolant $F$ is required to be nonnegative everywhere on $\mathbb{R}^n$.

For results related to Theorem \ref{Th4}, we refer the reader to our paper \cite{fil-2016} and references therein. 

In the following sections, we will set up the notation; then we will recall a main theorem in \cite{fil-2016} and use it to prove Theorem \ref{Th4}.

This paper is part of a literature on extension, interpolation, and
selection of functions, going back to H. Whitney's seminal work \cite%
{whitney-1934}, and including fundamental contributions by G. Glaeser \cite%
{gl-1958}, Y, Brudnyi and P. Shvartsman \cite%
{bs-1985, bs-1994,bs-1997,bs-1998,bs-2001, pavel-1982,pavel-1984,pavel-1986, pavel-1987, pavel-1992, s-2001,s-2002,pavel-2004, s-2008}, J. Wells 
\cite{jw-1973}, E. Le Gruyer \cite{lgw-2009}, and E. Bierstone, P. Milman,
and W. Paw{\l }ucki \cite{bmp-2003,bmp-2006,bm-2007}, as well as our own
papers \cite{f-2005, f-2006, f-2007, f-2009-b, fb1, fb2, f-2005-a, fl-2014}. See e.g. \cite{f-2009-b}
for the history of the problem, as well as Zobin \cite{zobin-1998,zobin-1999} for a related problem. 

We are grateful to the American Institute of Mathematics, the Banff
International Research Station, the Fields Institute, and the College of
William and Mary for hosting workshops on interpolation and extension. We
are grateful also to the Air Force Office of Scientific Research, the
National Science Foundation, the Office of Naval Research, and the U.S.-Israel Binational Science Foundation for financial support.

We are also grateful to Pavel Shvartsman and Alex Brudnyi for their comments on an earlier version of our manuscript, and to all the participants of the Eighth Whitney Problems Workshop for their interest in our work. 

\section{Notation and Preliminaries\label{notation-and-preliminaries}}
\subsection{Background Notation} \label{notation-prelim}
Fix $m$, $n\geq 1$. We will work with cubes in $\mathbb{R}^{n}$; all our
cubes have sides parallel to the coordinate axes. If $Q$ is a cube, then $%
\delta _{Q}$ denotes the sidelength of $Q$. For real numbers $A>0$, $AQ$
denotes the cube whose center is that of $Q$, and whose sidelength is $%
A\delta _{Q}$.

A \underline{dyadic} cube is a cube of the form $I_{1}\times I_{2}\times
\cdots \times I_{n}\subset \mathbb{R}^{n}$, where each $I_{\nu }$ has the
form $[2^{k}\cdot i_{\nu },2^{k}\cdot \left( i_{\nu }+1\right) )$ for
integers $i_{1},\cdots ,i_{n}$, $k$. Each dyadic cube $Q$ is contained in
one and only one dyadic cube with sidelength $2\delta _{Q}$; that cube is
denoted by $Q^{+}$.

We write $B_{n}\left( x,r\right) $ to denote the open ball in $\mathbb{R}%
^{n} $ with center $x$ and radius $r$, with respect to the Euclidean metric. 

We write $\mathcal{P}$ to denote the vector space of all real-valued
polynomials of degree at most $\left( m-1\right) $ on $\mathbb{R}^{n}$. If $%
x\in \mathbb{R}^{n}$ and $F$ is a real-valued $C^{m-1}$ function on a
neighborhood of $x$, then $J_{x}\left( F\right) $ (the \textquotedblleft
jet" of $F$ at $x$) denotes the $\left( m-1\right) ^{rst}$ order Taylor
polynomial of $F$ at $x$, i.e.,

\begin{equation*}
J_{x}\left( F\right) \left( y\right) =\sum_{\left\vert \alpha \right\vert
\leq m-1}\frac{1}{\alpha !}\partial ^{\alpha }F\left( x\right) \cdot \left(
y-x\right) ^{\alpha }.
\end{equation*}%
Thus, $J_{x}\left( F\right) \in \mathcal{P}$.

For each $x\in \mathbb{R}^{n}$, there is a natural multiplication $\odot
_{x} $ on $\mathcal{P}$ (\textquotedblleft multiplication of jets at $x$")
defined by setting%
\begin{equation*}
P\odot _{x}Q=J_{x}\left( PQ\right) \text{ for }P,Q\in \mathcal{P}\text{.}
\end{equation*}%

If $F$ is a real-valued function on a cube $Q$, then we write $F\in
C^{m}\left( Q\right) $ to denote that $F$ and its derivatives up to $m$-th
order extend continuously to the closure of $Q$. For $F\in C^{m}\left(
Q\right) $, we define 
\begin{equation*}
\left\Vert F\right\Vert _{C^{m}\left( Q\right) }=\sup_{x\in
Q}\max_{\left\vert \alpha \right\vert \leq m}\left\vert \partial ^{\alpha
}F\left( x\right) \right\vert .
\end{equation*}%

The function space $C^{m-1,1}(Q)$ and the norm $\Vert \cdot \Vert_{C^{m-1,1}(Q)}$ are defined analogously.

If $F\in C^{m}\left( Q\right) $ and $x$ belongs to the boundary of $Q$, then
we still write $J_{x}\left( F\right) $ to denote the $\left( m-1\right)
^{rst}$ degree Taylor polynomial of $F$ at $x$, even though $F$ isn't
defined on a full neighborhood of $x\in \mathbb{R}^{n}$.

Let $S\subset \mathbb{R}^{n}$ be non-empty and finite. A \underline{Whitney
field} on $S$ is a family of polynomials 
\begin{equation*}
\vec{P}=\left( P^{y}\right) _{y\in S}\text{ (each }P^{y}\in \mathcal{P}\text{%
),}
\end{equation*}%
parametrized by the points of $S$.

We write $Wh\left( S\right) $ to denote the vector space of all Whitney
fields on $S$.

For $\vec{P}=\left( P^{y}\right) _{y\in S}\in Wh\left( S\right) $, we define
the seminorm 
\begin{equation*}
\left\Vert \vec{P}\right\Vert _{\dot{C}^{m}\left( S\right) }=\max_{x,y\in
S,\left( x\not=y\right), |\alpha| \leq m}\frac{\left\vert \partial ^{\alpha
}\left( P^{x}-P^{y}\right) \left( x\right) \right\vert }{\left\vert
x-y\right\vert ^{m-\left\vert \alpha \right\vert }}\text{.}
\end{equation*}

(If $S$ consists of a single point, then $\left\Vert \vec{P}\right\Vert _{%
\dot{C}^{m}\left( S\right) }=0$.)

We also need an elementary fact about convex sets. 

\theoremstyle{plain} \newtheorem*{thm Helly}{Helly's Theorem}%
\begin{thm Helly}
Let $K_{1},\cdots ,K_{N}\subset \mathbb{R}^{D}$ be convex. Suppose that $K_{i_{1}}\cap \cdots \cap K_{i_{D+1}}$ is nonempty for any $i_{1}, \cdots, i_{D+1}\in
\{1,\cdots ,N\}$. Then $K_{1}\cap \cdots \cap K_{N}$ is nonempty.\end{thm Helly}

See \cite{rock-convex}.

\subsection{Shape Fields}

Let $E\subset \mathbb{R}^{n}$ be finite. For each $x\in E$, $M\in \left(
0,\infty \right) $, let $\Gamma \left( x,M\right) \subseteq \mathcal{P}$ be
a (possibly empty) convex set. We say that $\vec{\Gamma}=\left( \Gamma
\left( x,M\right) \right) _{x\in E,M>0}$ is a \underline{shape field} if for
all $x\in E$ and $0<M^{\prime }\leq M<\infty $, we have 
\begin{equation*}
\Gamma \left( x,M^{\prime }\right) \subseteq \Gamma \left( x,M\right) .
\end{equation*}

Let $\vec{\Gamma}=\left( \Gamma \left( x,M\right) \right) _{x\in E,M>0}$ be
a shape field and let $C_{w},\delta _{\max }$ be positive real numbers. We
say that $\vec{\Gamma}$ is \underline{$\left( C_{w},\delta _{\max }\right) $%
-convex} if the following condition holds:

Let $0<\delta \leq \delta _{\max }$, $x\in E$, $M\in \left( 0,\infty \right) 
$, $P_{1}$, $P_{2}$, $Q_{1}$, $Q_{2}\in \mathcal{P}$. Assume that

\begin{itemize}
\item[\refstepcounter{equation}\text{(\theequation)}\label{wsf1}] $P_1,P_2
\in\Gamma(x,M)$;

\item[\refstepcounter{equation}\text{(\theequation)}\label{wsf2}] $%
|\partial^\beta(P_1-P_2)(x)| \leq M\delta^{m-|\beta|}$ for $|\beta| \leq m-1$%
;

\item[\refstepcounter{equation}\text{(\theequation)}\label{wsf3}] $%
|\partial^\beta Q_i(x)|\leq \delta^{-|\beta|}$ for $|\beta| \leq m-1$ for $%
i=1,2$;

\item[\refstepcounter{equation}\text{(\theequation)}\label{wsf4}] $%
Q_1\odot_x Q_1 + Q_2 \odot_x Q_2 =1$.
\end{itemize}

Then

\begin{itemize}
\item[\refstepcounter{equation}\text{(\theequation)}\label{wsf5}] $%
P:=Q_1\odot_x Q_1\odot_x P_1 + Q_2 \odot_x Q_2 \odot_x P_2\in\Gamma(x,C_wM)$.
\end{itemize}

\subsection{Finiteness Principle for Shape Fields}
We recall a main result proven in \cite{fil-2016}. 
\begin{theorem}
\label{theorem-fp-for-wsf} For a large enough $k^{\#}$ determined by $m$, $n$%
, the following holds. Let $\vec{\Gamma}_{0}=\left( \Gamma _{0}\left(
x,M\right) \right) _{x\in E,M>0}$ be a $\left( C_{w},\delta _{\max }\right) $%
-convex shape field and let $Q_{0}\subset \mathbb{R}^{n}$ be a cube of
sidelength $\delta _{Q_{0}}\leq \delta _{\max }$. Also, let $x_{0}\in E\cap
5Q_{0}$ and $M_{0}>0$ be given. Assume that for each $S\subset E$ with $%
\#\left( S\right) \leq k^{\#}$ there exists a Whitney field $\vec{P}%
^{S}=\left( P^{z}\right) _{z\in S}$ such that 
\begin{equation*}
\left\Vert \vec{P}^{S}\right\Vert _{\dot{C}^{m}\left( S\right) }\leq M_{0}%
\text{,}
\end{equation*}%
and 
\begin{equation*}
P^{z}\in \Gamma _{0}\left( z,M_{0}\right) \text{ for all }z\in S\text{.}
\end{equation*}%
Then there exist $P^{0}\in \Gamma _{0}\left( x_{0},M_{0}\right) $ and $F\in
C^{m}\left( Q_{0}\right) $ such that the following hold, with a constant $%
C_{\ast }$ determined by $C_{w}$, $m$, $n$:

\begin{itemize}
\item $J_{z}(F)\in \Gamma _{0}\left( z,C_{\ast }M_{0}\right) $ for all $z\in
E\cap Q_{0}$.

\item $|\partial ^{\beta }\left( F-P^{0}\right) \left( x\right) |\leq
C_{\ast }M_{0}\delta _{Q_{0}}^{m-\left\vert \beta \right\vert }$ for all $%
x\in Q_{0}$, $\left\vert \beta \right\vert \leq m$.

\item In particular, $\left\vert \partial ^{\beta }F\left( x\right)
\right\vert \leq C_{\ast }M_{0}$ for all $x\in Q_{0}$, $\left\vert \beta
\right\vert =m$.
\end{itemize}
\end{theorem}

\section{$C^{m}$ Interpolation by Nonnegative Functions}

\label{inf}

In this section, $c$, $C$, $C^{\prime }$, etc. denote constants determined
by $m$ and $n$. These symbols may denote different constants in different
occurrences. For $x \in \mathbb{R}^n$ and $M >0$, define

\begin{itemize}
\item[\refstepcounter{equation}\text{(\theequation)}\label{inf1}] $\Gamma
_{\ast }\left( x,M\right) =\left\{ 
\begin{array}{c}
P\in \mathcal{P}:\text{ There exists }F\in C^{m}\left( \mathbb{R}^{n}\right) 
\text{ with }\left\Vert F\right\Vert _{C^{m}\left( \mathbb{R}^{n}\right)
}\leq M\text{,} \\ 
F\geq 0\text{ on }\mathbb{R}^{n}\text{, }J_{x}\left( F\right) =P\text{.}%
\end{array}%
\right\} $
\end{itemize}

It is not immediately clear how to compute $\Gamma _{\ast }$; we will return
to this issue in a later section. Let $E\subset \mathbb{R}^{n}$ be finite,
and let $f:E\rightarrow \lbrack 0,\infty )$. Define $\vec{\Gamma}%
_{f}=(\Gamma _{f}(x,M))_{x\in E,M>0}$, where

\begin{itemize}
\item[\refstepcounter{equation}\text{(\theequation)}\label{inf2}] $\Gamma
_{f}\left( x,M\right) =\left\{ P\in \Gamma _{\ast }\left( x,M\right)
:P\left( x\right) =f\left( x\right) \right\} $.
\end{itemize}

\begin{lemma}
\label{lemma-inf1} $\vec{\Gamma}_f$ is a $(C,1)$-convex shape field.
\end{lemma}

\begin{proof}
It is clear that $\vec{\Gamma}_f$ is a shape field, i.e., each $\Gamma_f(x,M)
$ is convex, and $M^{\prime }\leq M$ implies $\Gamma_f(x,M^{\prime })
\subseteq \Gamma_f(x,M)$. To establish $(C,1)$-convexity, suppose we are
given the following:

\begin{itemize}
\item[\refstepcounter{equation}\text{(\theequation)}\label{inf3}] $0<\delta
\leq 1$, $x\in E$, $M>0$;
\end{itemize}

\begin{itemize}
\item[\refstepcounter{equation}\text{(\theequation)}\label{inf4}] $%
P_{1},P_{2}\in \Gamma _{f}\left( x,M\right) $ satisfying
\end{itemize}

\begin{itemize}
\item[\refstepcounter{equation}\text{(\theequation)}\label{inf5}] $%
\left\vert \partial ^{\beta }\left( P_{1}-P_{2}\right) \left( x\right)
\right\vert \leq M\delta ^{m-\left\vert \beta \right\vert }$ for $\left\vert
\beta \right\vert \leq m-1$;
\end{itemize}

\begin{itemize}
\item[\refstepcounter{equation}\text{(\theequation)}\label{inf6}] $%
Q_{1},Q_{2}\in \mathcal{P}$ satisfying
\end{itemize}

\begin{itemize}
\item[\refstepcounter{equation}\text{(\theequation)}\label{inf7}] $%
\left\vert \partial ^{\beta }Q_{i}\left( x\right) \right\vert \leq \delta
^{-\left\vert \beta \right\vert }$ for $\left\vert \beta \right\vert \leq
m-1 $, $i=1,2$, and
\end{itemize}

\begin{itemize}
\item[\refstepcounter{equation}\text{(\theequation)}\label{inf8}] $%
Q_{1}\odot _{x}Q_{1}+Q_{2}\odot _{x}Q_{2}=1$.
\end{itemize}

Set

\begin{itemize}
\item[\refstepcounter{equation}\text{(\theequation)}\label{inf9}] $%
P=Q_{1}\odot _{x}Q_{1}\odot _{x}P_{1}+Q_{2}\odot _{x}Q_{2}\odot _{x}P_{2}$.
\end{itemize}

We must prove that

\begin{itemize}
\item[\refstepcounter{equation}\text{(\theequation)}\label{inf10}] $P\in
\Gamma _{f}\left( x,CM\right) $.
\end{itemize}

Thanks to \eqref{inf4}, we have

\begin{itemize}
\item[\refstepcounter{equation}\text{(\theequation)}\label{inf11}] $%
P_{1}\left( x\right) =f\left( x\right) $ and $P_{2}\left( x\right) =f\left(
x\right) $,
\end{itemize}

and there exist functions $F_1, F_2 \in C^m(\mathbb{R}^n)$ such that

\begin{itemize}
\item[\refstepcounter{equation}\text{(\theequation)}\label{inf12}] $%
\left\Vert F_{i}\right\Vert _{C^{m}\left( \mathbb{R}^{n}\right) }\leq M$ $%
\left( i=1,2\right) $,
\end{itemize}

\begin{itemize}
\item[\refstepcounter{equation}\text{(\theequation)}\label{inf13}] $%
F_{i}\geq 0$ on $\mathbb{R}^{n}$ $\left( i=1,2\right) $, and
\end{itemize}

\begin{itemize}
\item[\refstepcounter{equation}\text{(\theequation)}\label{inf14}] $%
J_{x}\left( F_{i}\right) =P_{i}$ $\left( i=1,2\right) $.
\end{itemize}

We fix $F_1$, $F_2$ as above. By \eqref{inf8}, we have $|Q_i(x)| \geq \frac{1%
}{\sqrt{2}}$ for $i =1$ or for $i=2$. By possibly interchanging $Q_1$ and $Q_2$, and then possibly changing $Q_1$ to $-Q_1$,
we may suppose that

\begin{itemize}
\item[\refstepcounter{equation}\text{(\theequation)}\label{inf15}] $%
Q_{1}\left( x\right) \geq \frac{1}{\sqrt{2}}$.
\end{itemize}

For small enough $c_0$, \eqref{inf7} and \eqref{inf15} yield

\begin{itemize}
\item[\refstepcounter{equation}\text{(\theequation)}\label{inf16}] $%
Q_{1}\left( y\right) \geq \frac{1}{10}$ for $\left\vert y-x\right\vert \leq
c_{0}\delta $.
\end{itemize}

Fix $c_0$ as in \eqref{inf16}. We introduce a $C^m$ cutoff function $\chi$
on $\mathbb{R}^n$ with the following properties.

\begin{itemize}
\item[\refstepcounter{equation}\text{(\theequation)}\label{inf17}] $0\leq
\chi \leq 1$ on $\mathbb{R}^{n}$; $\chi =0$ outside $B_{n}\left(
x,c_{0}\delta \right) $; $\chi =1$ in a neighborhood of $x$;
\end{itemize}

\begin{itemize}
\item[\refstepcounter{equation}\text{(\theequation)}\label{inf18}] $%
\left\vert \partial ^{\beta }\chi \right\vert \leq C\delta ^{-\left\vert
\beta \right\vert }$ on $\mathbb{R}^{n}$, for $\left\vert \beta \right\vert
\leq m$.
\end{itemize}

We then define $\tilde{\theta}_{1}=\chi \cdot Q_{1}+\left( 1-\chi \right) $
and $\tilde{\theta}_{2}=\chi \cdot Q_{2}$.

These functions satisfy the following: $\tilde{\theta}_{i}\in C^{m}\left( 
\mathbb{R}^{n}\right) $ and $\left\vert \partial ^{\beta }\tilde{\theta}%
_{i}\right\vert \leq C\delta ^{-\left\vert \beta \right\vert }$ on $\mathbb{R%
}^{n}$ for $\left\vert \beta \right\vert \leq m$, $i=1,2$; $\tilde{\theta}%
_{1}\geq \frac{1}{10}$ on $\mathbb{R}^{n}$; $J_{x}\left( \tilde{\theta}%
_{i}\right) =Q_{i}$ for $i=1,2$; outside $B_{n}\left( x,c_{0}\delta \right) $
we have $\tilde{\theta}_{1}=1$ and $\tilde{\theta}_{2}=0$. Setting $\theta
_{i}=\tilde{\theta}_{i}\cdot \left( \tilde{\theta}_{1}^{2}+\tilde{\theta}%
_{2}^{2}\right) ^{-1/2}$ for $i=1$, $2$, we find that

\begin{itemize}
\item[\refstepcounter{equation}\text{(\theequation)}\label{inf19}] $\theta
_{i}\in C^{m}\left( \mathbb{R}^{n}\right) $ and $\left\vert \partial ^{\beta
}\theta _{i}\right\vert \leq C\delta ^{-\left\vert \beta \right\vert }$ on $%
\mathbb{R}^{n}$ for $\left\vert \beta \right\vert \leq m$, $i=1,2$;
\end{itemize}

\begin{itemize}
\item[\refstepcounter{equation}\text{(\theequation)}\label{inf20}] $\theta
_{1}^{2}+\theta _{2}^{2}= 1$ on $\mathbb{R}^{n}$;
\end{itemize}

\begin{itemize}
\item[\refstepcounter{equation}\text{(\theequation)}\label{inf21}] $%
J_{x}\left( \theta _{i}\right) =Q_{i}$ for $i=1,2$ (here we use (\ref{inf8}%
)); and
\end{itemize}

\begin{itemize}
\item[\refstepcounter{equation}\text{(\theequation)}\label{inf22}] outside $%
B_{n}\left( x,c_{0}\delta \right) $ we have $\theta _{1}=1$ and $\theta
_{2}=0$.
\end{itemize}

Now set

\begin{itemize}
\item[\refstepcounter{equation}\text{(\theequation)}\label{inf23}] $F=\theta
_{1}^{2}F_{1}+\theta _{2}^{2}F_{2}=F_{1}+\theta _{2}^{2}\left(
F_{2}-F_{1}\right) $ (see (\ref{inf20})).
\end{itemize}

Clearly $F\in C^{m}(\mathbb{R}^{n})$. By \eqref{inf14}, we have $%
J_{x}(F_{2}-F_{1})=P_{2}-P_{1}$; hence \eqref{inf5} yields the estimate%
\begin{equation*}
\left\vert \partial ^{\beta }\left( F_{2}-F_{1}\right) \left( x\right)
\right\vert \leq CM\delta ^{m-\left\vert \beta \right\vert }\text{ for }%
\left\vert \beta \right\vert \leq m-1\text{.}
\end{equation*}

Together with \eqref{inf12}, this tells us that%
\begin{equation*}
\left\vert \partial ^{\beta }\left( F_{2}-F_{1}\right) \right\vert \leq
CM\delta ^{m-\left\vert \beta \right\vert }\text{ on }B_{n}\left(
x,c_{0}\delta \right) \text{ for }\left\vert \beta \right\vert \leq m\text{.}
\end{equation*}

Recalling \eqref{inf19}, we deduce that 
\begin{equation*}
\left\vert \partial ^{\beta }\left( \theta _{2}^{2}\cdot \left(
F_{2}-F_{1}\right) \right) \right\vert \leq CM\delta ^{m-\left\vert \beta
\right\vert }\text{ on }B_{n}\left( x,c_{0}\delta \right) \text{ for }%
\left\vert \beta \right\vert \leq m\text{.}
\end{equation*}%
Together with \eqref{inf12} and \eqref{inf23}, this implies
that 
\begin{equation*}
\left\vert \partial ^{\beta }F\right\vert \leq CM\text{ on }B_{n}\left(
x,c_{0}\delta \right) \text{,}
\end{equation*}%
since $0<\delta \leq 1$ (see \eqref{inf3}). On the other hand, outside $%
B_{n}(x,c_{0}\delta )$ we have $F=F_{1}$ by \eqref{inf22}, \eqref{inf23};
hence $|\partial ^{\beta }F|\leq CM$ outside $B_{n}(x,c_{0}\delta )$ for $%
|\beta |\leq m$, by \eqref{inf12}. Thus, $|\partial ^{\beta }F|\leq CM$ on
all of $\mathbb{R}^{n}$ for $|\beta |\leq m$, i.e.,

\begin{itemize}
\item[\refstepcounter{equation}\text{(\theequation)}\label{inf24}] $%
\left\Vert F\right\Vert _{C^{m}\left( \mathbb{R}^{n}\right) }\leq CM$.
\end{itemize}

Also, from \eqref{inf13} and \eqref{inf23} we have

\begin{itemize}
\item[\refstepcounter{equation}\text{(\theequation)}\label{inf25}] $F\geq 0$
on $\mathbb{R}^{n}$;
\end{itemize}

and \eqref{inf9}, \eqref{inf14}, \eqref{inf21}, \eqref{inf23}
imply that

\begin{itemize}
\item[\refstepcounter{equation}\text{(\theequation)}\label{inf26}] $%
J_{x}\left( F\right) =Q_{1}\odot _{x}Q_{1}\odot _{x}P_{1}+Q_{2}\odot
_{x}Q_{2}\odot _{x}P_{2}=P$.
\end{itemize}

Since $F\in C^{m}\left( \mathbb{R}^{n}\right) $ satisfies \eqref{inf24}, %
\eqref{inf25}, \eqref{inf26}, we have

\begin{itemize}
\item[\refstepcounter{equation}\text{(\theequation)}\label{inf27}] $P\in
\Gamma _{\ast }\left( x,CM\right) $.
\end{itemize}

Moreover,

\begin{itemize}
\item[\refstepcounter{equation}\text{(\theequation)}\label{inf28}] $P\left(
x\right) =\left( Q_{1}\left( x\right) \right) ^{2}f\left( x\right) +\left(
Q_{2}\left( x\right) \right) ^{2}f\left( x\right) =f\left( x\right) $,
\end{itemize}
thanks to \eqref{inf8}, \eqref{inf9}, \eqref{inf11}. 

From \eqref{inf27}, %
\eqref{inf28} we conclude that $P \in \Gamma_f(x,CM)$, completing the proof
of Lemma \ref{lemma-inf1}.
\end{proof}

\begin{lemma}
\label{lemma-inf2} Let $\left( P^{x}\right) _{x\in E}$ be a Whitney field on
the finite set $E$, and let $M>0$. Suppose that

\begin{itemize}
\item[\refstepcounter{equation}\text{(\theequation)}\label{inf29}] $P^{x}\in
\Gamma _{\ast }\left( x,M\right) $ for each $x\in E$,
\end{itemize}

and that

\begin{itemize}
\item[\refstepcounter{equation}\text{(\theequation)}\label{inf30}] $%
\left\vert \partial ^{\beta }\left( P^{x}-P^{x^{\prime }}\right) \left(
x\right) \right\vert \leq M\left\vert x-x^{\prime }\right\vert
^{m-\left\vert \beta \right\vert }$ for $x,x^{\prime }\in E$ and $\left\vert
\beta \right\vert \leq m-1$.
\end{itemize}

Then there exists $F\in C^{m}(\mathbb{R}^{n})$ such that

\begin{itemize}
\item[\refstepcounter{equation}\text{(\theequation)}\label{inf31}] $%
\left\Vert F\right\Vert _{C^{m}\left( \mathbb{R}^{n}\right) }\leq CM$,
\end{itemize}

\begin{itemize}
\item[\refstepcounter{equation}\text{(\theequation)}\label{inf32}] $F\geq 0$
on $\mathbb{R}^{n}$, and
\end{itemize}

\begin{itemize}
\item[\refstepcounter{equation}\text{(\theequation)}\label{inf33}] $%
J_{x}\left( F\right) =P^{x}$ for all $x\in E$.
\end{itemize}
\end{lemma}

\begin{proof}
We modify slightly Whitney's proof \cite{whitney-1934} of the Whitney
extension theorem. We say that a dyadic cube $Q \subset \mathbb{R}^n$ is
``OK'' if $\#(E\cap 5Q) \leq 1$ and $\delta_Q \leq 1$. Then every small
enough $Q$ is OK (because $E$ is finite), and no $Q$ of sidelength $\delta_Q
>1$ is OK. Also, let $Q, Q^{\prime }$ be dyadic cubes with $5Q \subset
5Q^{\prime }$. If $Q^{\prime }$ is OK, then also $Q$ is OK. We define a 
\underline{Calder\'on-Zygmund} (or \underline{CZ}) cube to be an OK cube $Q$
such that no $Q^{\prime }$ that strictly contains $Q$ is OK. The above
remarks imply that the CZ cubes form a partition of $\mathbb{R}^n$; that the
sidelengths of the CZ cubes are bounded above by $1$ and below by some
positive number; and that the following condition holds.

\begin{itemize}
\item[\refstepcounter{equation}\text{(\theequation)}\label{inf34}] %
\textquotedblleft Good Geometry": If $Q,Q^{\prime }\in $ CZ and $\frac{65}{64%
}Q\cap \frac{65}{64}Q^{\prime }\not=\emptyset $, then $\frac{1}{2}\delta
_{Q}\leq \delta _{Q^{\prime }}\leq 2\delta _{Q}$.
\end{itemize}

We classify CZ cubes into three types as follows. 

$Q \in CZ$ is of

\begin{description}
\item[Type 1] if $E \cap 5Q \not= \emptyset$

\item[Type 2] if $E \cap 5Q =\emptyset$ and $\delta_Q <1$.

\item[Type 3] if $E \cap 5Q =\emptyset$ and $\delta_Q=1$.
\end{description}

\underline{Let $Q\in $ CZ be of Type 1.} Since $Q$ is OK, we have $\#(E\cap
5Q)\leq 1$. Hence $E\cap 5Q$ is a singleton, $E\cap 5Q=\left\{ x_{Q}\right\} 
$. Since $P^{x_{Q}}\in \Gamma _{\ast }\left( x_{Q},M\right) $, there exists $%
F_{Q}\in C^{m}\left( \mathbb{R}^{n}\right) $ such that

\begin{itemize}
\item[\refstepcounter{equation}\text{(\theequation)}\label{inf35}] $%
\left\Vert F_{Q}\right\Vert _{C^{m}\left( \mathbb{R}^{n}\right) }\leq M$, $%
F_{Q}\geq 0$ on $\mathbb{R}^{n}$, $J_{x_{Q}}\left( F_{Q}\right) =P^{x_{Q}}$.
\end{itemize}

We fix $F_Q$ as in \eqref{inf35}.

\underline{Let $Q\in $ CZ be of Type 2.} Then $\delta _{Q^{+}}\leq 1$ but $%
Q^{+}$ is not OK; hence $\#\left( E\cap 5Q^{+}\right) \geq 2$. We pick $%
x_{Q}\in E\cap 5Q^{+}$. Since $P^{x_{Q}}\in \Gamma _{\ast }\left(
x_{Q},M\right) $, there exists $F_{Q}\in C^{m}\left( \mathbb{R}^{n}\right) $
satisfying \eqref{inf35}. We fix such an $F_{Q}$.

\underline{Let $Q \in $ CZ be of Type 3.} Then we set $F_Q =0$. In place of %
\eqref{inf35}, we have the trivial results

\begin{itemize}
\item[\refstepcounter{equation}\text{(\theequation)}\label{inf36}] $%
\left\Vert F_{Q}\right\Vert _{C^{m}\left( \mathbb{R}^{n}\right) }=0$ and $%
F_{Q}\geq 0$ on $\mathbb{R}^{n}$.
\end{itemize}

Thus, we have defined $F_{Q}$ for all $Q\in $ CZ, and we have defined $%
x_{Q}\in E\cap 5Q^{+}$ for all $Q$ of Type 1 or Type 2. Note that

\begin{itemize}
\item[\refstepcounter{equation}\text{(\theequation)}\label{inf37}] $%
J_{x}\left( F_{Q}\right) =P^{x}$ for all $x\in E\cap 5Q$.
\end{itemize}

Indeed, if $Q$ is of Type 1, then \eqref{inf37} follows from \eqref{inf35}
since $E\cap 5Q=\{x_{Q}\}$. If $Q$ is of Type 2 or Type 3, then \eqref{inf37}
holds vacuously since $E\cap 5Q=\emptyset $. Now suppose $Q,Q^{\prime }\in $
CZ and $\frac{65}{64}Q\cap \frac{65}{64}Q^{\prime }\not=\emptyset $. We will
show that

\begin{itemize}
\item[\refstepcounter{equation}\text{(\theequation)}\label{inf38}] $%
\left\vert \partial ^{\beta }\left( F_{Q}-F_{Q^{\prime }}\right) \right\vert
\leq CM\delta _{Q}^{m-\left\vert \beta \right\vert }$ on $\frac{65}{64}Q\cap 
\frac{65}{64}Q^{\prime }$ for $\left\vert \beta \right\vert \leq m$.
\end{itemize}

To see this, suppose first that $Q$ or $Q^{\prime }$ is of Type 3. Then $%
\delta_Q$ or $\delta_{Q^{\prime }}$ is equal to $1$, hence $\delta_Q \geq 
\frac{1}{2}$ by \eqref{inf34}. Consequently, \eqref{inf38} asserts simply
that

\begin{itemize}
\item[\refstepcounter{equation}\text{(\theequation)}\label{inf39}] $%
\left\vert \partial ^{\beta }\left( F_{Q}-F_{Q^{\prime }}\right) \right\vert
\leq CM$ on $\frac{65}{64}Q\cap \frac{65}{64}Q^{\prime }$ for $\left\vert
\beta \right\vert \leq m$,
\end{itemize}
and \eqref{inf39} follows at once from \eqref{inf35}, \eqref{inf36}. Thus, %
\eqref{inf38} holds if $Q$ or $Q^{\prime }$ is of Type 3. Suppose that
neither $Q$ nor $Q^{\prime }$ is of Type 3. Then $x_{Q}\in E\cap 5Q^{+}$, $%
x_{Q^{\prime }}\in E\cap 5(Q^{\prime +})$, $\frac{65}{64}Q\cap \frac{65}{64}%
Q^{\prime }\not=\emptyset $, $\frac{1}{2}\delta _{Q}\leq \delta _{Q^{\prime
}}\leq 2\delta _{Q}$. Consequently,

\begin{itemize}
\item[\refstepcounter{equation}\text{(\theequation)}\label{inf40}] $%
\left\vert x_{Q}-x_{Q^{\prime }}\right\vert \leq C\delta _{Q}$, and
\end{itemize}

\begin{itemize}
\item[\refstepcounter{equation}\text{(\theequation)}\label{inf41}] $%
\left\vert x-x_{Q}\right\vert $, $\left\vert x-x_{Q^{\prime }}\right\vert
\leq C\delta _{Q}$ for all $x\in \frac{65}{64}Q\cap \frac{65}{64}Q^{\prime }$%
.
\end{itemize}

Applying \eqref{inf35} to $Q$ and to $Q^{\prime }$, we find that

\begin{itemize}
\item[\refstepcounter{equation}\text{(\theequation)}\label{inf42}] $%
\left\vert \partial ^{\beta }\left( F_{Q}-P^{x_{Q}}\right) \left( x\right)
\right\vert \leq CM\left\vert x-x_{Q}\right\vert ^{m-\left\vert \beta
\right\vert }\leq CM\delta _{Q}^{m-\left\vert \beta \right\vert }$, and
\end{itemize}

\begin{itemize}
\item[\refstepcounter{equation}\text{(\theequation)}\label{inf43}] $%
\left\vert \partial ^{\beta }\left( F_{Q^{\prime }}-P^{x_{Q^{\prime
}}}\right) \left( x\right) \right\vert \leq CM\left\vert x-x_{Q^{\prime
}}\right\vert ^{m-\left\vert \beta \right\vert }\leq CM\delta
_{Q}^{m-\left\vert \beta \right\vert }$,
\end{itemize}

for $x\in \frac{65}{64}Q\cap \frac{65}{64}Q^{\prime }$, $\left\vert \beta
\right\vert \leq m$. \newline
Also, \eqref{inf30}, \eqref{inf40}, \eqref{inf41} imply that

\begin{itemize}
\item[\refstepcounter{equation}\text{(\theequation)}\label{inf44}] $%
\left\vert \partial ^{\beta }\left( P^{x_{Q}}-P^{x_{Q^{\prime }}}\right)
\left( x\right) \right\vert \leq CM\delta _{Q}^{m-\left\vert \beta
\right\vert }$ for $x\in \frac{65}{64}Q\cap \frac{65}{64}Q^{\prime }$, $%
\left\vert \beta \right\vert \leq m$.
\end{itemize}

(Recall, $P^{x_{Q}}-P^{x_{Q^{\prime }}}$ is a polynomial of degree at most $%
m-1$.)

Estimates \eqref{inf42}, \eqref{inf43}, \eqref{inf44} together imply %
\eqref{inf38} in case neither $Q$ nor $Q^{\prime }$ is of Type 3. Thus, %
\eqref{inf38} holds in all cases.

Next, as in Whitney \cite{whitney-1934}, we introduce a partition of unity

\begin{itemize}
\item[\refstepcounter{equation}\text{(\theequation)}\label{inf45}] $%
1=\sum_{Q\in CZ}\theta _{Q}$ on $\mathbb{R}^{n}$,
\end{itemize}

where each $\theta_Q \in C^m(\mathbb{R}^n)$, and

\begin{itemize}
\item[\refstepcounter{equation}\text{(\theequation)}\label{inf46}] support $%
\theta _{Q}\subset \frac{65}{64}Q$, $\left\vert \partial ^{\beta }\theta
_{Q}\right\vert \leq C\delta _{Q}^{-\left\vert \beta \right\vert }$ for $%
\left\vert \beta \right\vert \leq m$, $\theta _{Q}\geq 0$ on $\mathbb{R}^{n}$%
.
\end{itemize}

We define

\begin{itemize}
\item[\refstepcounter{equation}\text{(\theequation)}\label{inf47}] $%
F=\sum_{Q\in CZ}\theta _{Q}F_{Q}$ on $\mathbb{R}^{n}$.
\end{itemize}

Thus, $F \in C^m_{loc}(\mathbb{R}^n)$ since CZ is a locally finite partition
of $\mathbb{R}^n$, and $F \geq 0$ on $\mathbb{R}^n$ since $\theta_Q \geq 0$
and $F_Q \geq 0$ for each $Q$. Let $\hat{x} \in \mathbb{R}^n$, and let $\hat{%
Q}$ be the one and only CZ cube containing $\hat{x}$. Then for $|\beta| \leq
m$, we have

\begin{itemize}
\item[\refstepcounter{equation}\text{(\theequation)}\label{inf48}] $\partial
^{\beta }F\left( \hat{x}\right) =\partial ^{\beta }F_{\hat{Q}}\left( \hat{x}%
\right) +\sum_{Q\in CZ}\partial ^{\beta }\left( \theta _{Q}\cdot \left(
F_{Q}-F_{\hat{Q}}\right) \right) \left( \hat{x}\right) $.
\end{itemize}

A given $Q\in $ CZ enters into the sum in \eqref{inf48} only if $\hat{x}\in 
\frac{65}{64}Q$; there are at most $C$ such cubes $Q$, thanks to %
\eqref{inf34}. Moreover, for each $Q\in $ CZ with $\hat{x}\in \frac{65}{64}Q$%
, we learn from \eqref{inf38} and \eqref{inf46} that 
\begin{equation*}
\left\vert \partial ^{\beta }\left( \theta _{Q}\cdot \left( F_{Q}-F_{\hat{Q}%
}\right) \right) \left( \hat{x}\right) \right\vert \leq CM\delta
_{Q}^{m-\left\vert \beta \right\vert }\leq CM\text{ for }\left\vert \beta
\right\vert \leq m\text{, since }\delta _{Q}\leq 1\text{.}
\end{equation*}%
Since also $\left\vert \partial ^{\beta }F_{\hat{Q}}\left( \hat{x}\right)
\right\vert \leq CM$ for $\left\vert \beta \right\vert \leq m$ by %
\eqref{inf35}, \eqref{inf36}, it now follows from \eqref{inf48} that $%
\left\vert \partial ^{\beta }F\left( \hat{x}\right) \right\vert \leq CM$ for
all $\left\vert \beta \right\vert \leq m$. Here, $\hat{x}\in \mathbb{R}^{n}$
is arbitrary. Thus, $F\in C^{m}\left( \mathbb{R}^{n}\right) $ and $%
||F||_{C^{m}\left( \mathbb{R}^{n}\right) }\leq CM$.

Next, let $x\in E$. For any $Q\in $ CZ such that $x\in \frac{65}{64}Q$, we
have $J_{x}(F_{Q})=P^{x}$, by \eqref{inf37}. Since support $\theta
_{Q}\subset \frac{65}{64}Q$ for each $Q\in $ CZ, it follows that $%
J_{x}(\theta _{Q}F_{Q})=J_{x}(\theta _{Q})\odot _{x}P^{x}$ for each $Q\in $
CZ, and consequently, 
\begin{equation*}
J_{x}(F)=\sum_{Q\in CZ}J_{x}\left( \theta _{Q}F_{Q}\right) =\left[
\sum_{Q\in CZ}J_{x}\left( \theta _{Q}\right) \right] \odot _{x}P^{x}=P^{x}%
\text{, by (\ref{inf45}).}
\end{equation*}%
Thus, $F\in C^{m}\left( \mathbb{R}^{n}\right) $, $\left\Vert F\right\Vert
_{C^{m}\left( \mathbb{R}^{n}\right) }\leq CM$, $F\geq 0$ on $\mathbb{R}^{n}$%
, and $J_{x}\left( F\right) =P^{x}$ for each $x\in E$.

The proof of Lemma \ref{lemma-inf2} is complete.
\end{proof}

\begin{theorem}[Finiteness Principle for Nonnegative $C^m$ Interpolation]
\label{theorem-fp-for-nonnegative-interpolation}

There exist constants $k^\#$, $C$, depending only on $m$, $n$, such that the
following holds.

Let $E \subset \mathbb{R}^n$ be finite, and let $f:E \rightarrow [0, \infty)$%
. Let $M_0 >0$. Suppose that for each $S \subset E$ with $\#(S) \leq k^\#$,
there exists $\vec{P}^S = (P^x)_{x\in S} \in Wh(S)$ such that

\begin{itemize}
\item $P^x \in \Gamma_f(x,M_0)$ for each $x \in S$, and

\item $|\partial^\beta (P^x -P^y)(x)| \leq M_0 |x-y|^{m-|\beta|}$ for $x, y
\in S$, $|\beta| \leq m-1$.
\end{itemize}

Then there exists $F\in C^{m}(\mathbb{R}^{n})$ such that

\begin{itemize}
\item $\|F\|_{C^m(\mathbb{R}^n)} \leq CM_0$,

\item $F \geq 0$ on $\mathbb{R}^n$, and

\item $F = f$ on $E$.
\end{itemize}
\end{theorem}

\begin{proof}
Suppose first that $E \subset \frac{1}{2}Q_0$ for a cube $Q_0$ of sidelength 
$\delta_{Q_0} =1$. Pick any $x_0 \in E$. (If $E$ is empty, our theorem holds
trivially.)

Let $S \subset E$ with $\#(S) \leq k^\#$.

Our present hypotheses supply the Whitney field $\vec{P}^S$ required in the hypotheses of Theorem \ref{theorem-fp-for-wsf}.

Hence, recalling Lemma \ref{lemma-inf1} and applying Theorem \ref{theorem-fp-for-wsf}, we obtain

\begin{itemize}
\item[\refstepcounter{equation}\text{(\theequation)}\label{inf49}] $P^0 \in
\Gamma_f(x_0, CM_0)$
\end{itemize}

and

\begin{itemize}
\item[\refstepcounter{equation}\text{(\theequation)}\label{inf50}] $F^0 \in
C^m(Q_0)$
\end{itemize}

such that

\begin{itemize}
\item[\refstepcounter{equation}\text{(\theequation)}\label{inf51}] $J_x(F^0)
\in \Gamma_f(x,CM_0)$ for all $x \in E \cap Q_0 =E$
\end{itemize}

and

\begin{itemize}
\item[\refstepcounter{equation}\text{(\theequation)}\label{inf52}] $%
|\partial^\beta(P^0 -F^0)|\leq CM_0$ on $Q_0$, for $|\beta|\leq m$.
\end{itemize}

From \eqref{inf1}, \eqref{inf2}, \eqref{inf49}, we have $|\partial^\beta P^0
(x_0)| \leq CM_0$ for $|\beta| \leq m-1$.

Since $P^0$ is a polynomial of degree at most $m-1$, and since $x_0 \in E
\subset Q_0$ with $\delta_{Q_0} =1$, it follows that $|\partial^\beta P^0 |
\leq CM_0$ on $Q_0$ for $|\beta| \leq m$.

Together with \eqref{inf52}, this tells us that

\begin{itemize}
\item[\refstepcounter{equation}\text{(\theequation)}\label{inf53}] $%
|\partial^\beta F^0| \leq CM_0$ on $Q_0$ for $|\beta| \leq m$.
\end{itemize}

Note that $F^0$ needn't be nonnegative.

Set $P^x = J_x(F^0)$ for $x \in E$. Then

\begin{itemize}
\item[\refstepcounter{equation}\text{(\theequation)}\label{inf54}] $P^{x}\in
\Gamma _{f}\left( x,CM_{0}\right) $ for $x\in E$, and
\end{itemize}

\begin{itemize}
\item[\refstepcounter{equation}\text{(\theequation)}\label{inf55}] $%
\left\vert \partial ^{\beta }\left( P^{x}-P^{y}\right) \left( x\right)
\right\vert \leq CM_{0}\left\vert x-y\right\vert ^{m-\left\vert \beta
\right\vert }$ for $x,y\in E$, $\left\vert \beta \right\vert \leq m-1$.
\end{itemize}

By Lemma \ref{lemma-inf2}, there exists $F\in C^{m}\left( \mathbb{R}%
^{n}\right) $ such that

\begin{itemize}
\item[\refstepcounter{equation}\text{(\theequation)}\label{inf56}] $%
\left\Vert F\right\Vert _{C^{m}\left( \mathbb{R}^{n}\right) }\leq CM_0$,
\end{itemize}

\begin{itemize}
\item[\refstepcounter{equation}\text{(\theequation)}\label{inf57}] $F\geq 0$
on $\mathbb{R}^{n}$, and
\end{itemize}

\begin{itemize}
\item[\refstepcounter{equation}\text{(\theequation)}\label{inf58}] $%
J_{x}\left( F\right) =P^{x}$ for each $x\in E$.
\end{itemize}

From \eqref{inf54} and \eqref{inf2}, we have $P^x(x)=f(x)$ for each $x \in E$%
; hence, \eqref{inf58} implies that

\begin{itemize}
\item[\refstepcounter{equation}\text{(\theequation)}\label{inf59}] $F\left(
x\right) =f\left( x\right) $ for each $x\in E$.
\end{itemize}

Our results \eqref{inf56}, \eqref{inf57}, \eqref{inf59} are the conclusions
of our theorem. Thus, we have proven Theorem \ref%
{theorem-fp-for-nonnegative-interpolation} in the case in which $E \subset 
\frac{1}{2} Q_0$ with $\delta_{Q_0} =1$.

To pass to the general case (arbitrary finite $E \subset \mathbb{R}^n$), we
set up a partition of unity $1 =\sum_{\nu} \chi_\nu$ on $\mathbb{R}^n$,
where each $\chi_\nu \in C^m(\mathbb{R}^n)$ and $\chi_\nu \geq 0$ on $%
\mathbb{R}^n$, $\|\chi_\nu\|_{C^m(\mathbb{R}^n)} \leq C$, support $\chi_\nu
\subset \frac{1}{2} Q_\nu$, with $\delta_{Q_\nu} =1$, and with any given
point of $\mathbb{R}^n$ belonging to at most $C$ of the $Q_\nu$.

For each $\nu$, we apply the known special case of our theorem to the set $%
E_\nu = E \cap \frac{1}{2}Q_\nu$ and the function $f_\nu = f|_{E_\nu}$.
Thus, we obtain $F_\nu \in C^m(\mathbb{R}^n)$, with $\|F_\nu\|_{C^m(\mathbb{R%
}^n)} \leq CM_0$, $F_\nu \geq 0$ on $\mathbb{R}^n$, and $F_\nu = f$ on $E
\cap \frac{1}{2} Q_\nu$.

Setting $F = \sum _\nu \chi_\nu F_\nu \in C^m_{loc}(\mathbb{R}^n)$, we
verify easily that $F \in C^m (\mathbb{R}^n)$, $\|F\|_{C^m(\mathbb{R}^n)}
\leq CM_0$, $F \geq 0$ on $\mathbb{R}^n$, and $F = f $ on $E$.

This completes the proof of Theorem \ref%
{theorem-fp-for-nonnegative-interpolation}.
\end{proof}

\begin{remark}
Conversely, we make the following trivial observation: Let $E\subset \mathbb{%
R}^{n}$ be finite, let $f:E\rightarrow \lbrack 0,\infty )$, and let $M_{0}>0$%
. Suppose $F\in C^{m}(\mathbb{R}^{n})$ satisfies $\Vert F\Vert _{C^{m}(%
\mathbb{R}^{n})}\leq M_{0}$, $F\geq 0$ on $\mathbb{R}^{n}$, $F=f$ on $E$.
Then for each $x\in E$, we have

\begin{itemize}
\item $P^{x}=J_{x}(F)\in \Gamma _{f}(x,M_{0})$ by \eqref{inf1}, \eqref{inf2}%
; and

\item $|\partial ^{\beta }(P^{x}-P^{y})(x)|\leq CM_{0}|x-y|^{m-|\beta |}$
for $x,y\in E$, $|\beta |\leq m-1$.
\end{itemize}

Therefore, for any $S \subset E$, the Whitney field $\vec{P}^S = (P^x)_{x
\in S} \in Wh(S)$ satisfies

\begin{itemize}
\item $P^{x}\in \Gamma _{f}(x,CM_{0})$ for $x\in S$, and

\item $|\partial ^{\beta }(P^{x}-P^{y})(x)|\leq CM_{0}|x-y|^{m-|\beta |}$
for $x,y\in S$, $|\beta |\leq m-1$.
\end{itemize}
Note that Theorem \ref{Th4} (A) follows easily from Theorem \ref{theorem-fp-for-nonnegative-interpolation}.
\end{remark}

\section{Computable Convex Sets}

\label{ccs}

In this section, we discuss computational issues regarding the convex set

\begin{itemize}
\item[\refstepcounter{equation}\text{(\theequation)}\label{cs1}] $\Gamma
_{\ast }\left( x,M\right) =\left\{ J_{x}\left( F\right) :F\in C^{m}\left( 
\mathbb{R}^{n}\right) \text{, }\left\Vert F\right\Vert _{C^{m}\left( \mathbb{%
R}^{n}\right) }\leq M\text{, }F\geq 0\text{ on }\mathbb{R}^{n}\right\} .$
\end{itemize}

We write $c$, $C$, $C^{\prime }$, etc., to denote constants determined by $m$
and $n$. These symbols may denote different constants in different
occurrences.

We will define convex sets $\tilde{\Gamma}_{\ast}(x,M) \subset \mathcal{P}$,
prove that

\begin{itemize}
\item[\refstepcounter{equation}\text{(\theequation)}\label{cs2}] $\tilde{%
\Gamma}_{\ast }(x,cM)\subset \Gamma _{\ast }\left( x,M\right) \subset \tilde{%
\Gamma}_{\ast }(x,CM)$ for all $x\in \mathbb{R}^{n}$, $M>0$,
\end{itemize}

and explain how (in principle) one can compute $\tilde{\Gamma}_{\ast}(x,M)$.

We may then use

\begin{itemize}
\item[\refstepcounter{equation}\text{(\theequation)}\label{cs3}] $\tilde{%
\Gamma}_{f}\left( x,M\right) =\left\{ P\in \tilde{\Gamma}_{\ast
}(x,M):P\left( x\right) =f\left( x\right) \right\} $
\end{itemize}
in place of $\Gamma_f(x,M)$ in the statement of Theorem \ref%
{theorem-fp-for-nonnegative-interpolation}. (The assertion in terms of $%
\tilde{\Gamma}_f$ follows trivially from \eqref{cs2} and the original
assertion in terms of $\Gamma_f$.)

To achieve \eqref{cs2}, we will define

\begin{itemize}
\item[\refstepcounter{equation}\text{(\theequation)}\label{cs4}] $\tilde{%
\Gamma}_{\ast }(x,M)=\left\{ MP\left( \cdot +x\right) ):P\in \tilde{\Gamma}%
_{0}\right\} $, for a convex set $\tilde{\Gamma}_{0}$.
\end{itemize}

We will prove that

\begin{itemize}
\item[\refstepcounter{equation}\text{(\theequation)}\label{cs5}] $%
\Gamma_{\ast}(0,c) \subset \tilde{\Gamma}_0 \subset \Gamma_{\ast}(0,C)$.
\end{itemize}

Property \eqref{cs2} then follows at once from \eqref{cs1}, \eqref{cs4}, and %
\eqref{cs5}.

Thus, our task is to define a convex set $\tilde{\Gamma}_{0}$ satisfying %
\eqref{cs5}, and explain how (in principle) one can compute $\tilde{\Gamma}%
_{0}$.

Recall that $\mathcal{P}$ is the vector space of $(m-1)$-jets. We will work
in the space of $m$-jets. In this section, we let $\mathcal{P}^{+}$ denote
the vector space of real-valued polynomials of degree at most $m$ on $%
\mathbb{R}^{n}$, and we write $J_{x}^{+}(F)$ to denote the $m^{th}$-degree
Taylor polynomial of $F$ at $x$, i.e.,%
\begin{equation*}
J_{x}^{+}\left( F\right) \left( y\right) =\sum_{\left\vert \alpha
\right\vert \leq m}\frac{1}{\alpha !}\left( \partial ^{\alpha }F\left(
x\right) \right) \cdot \left( y-x\right) ^{\alpha }\text{.}
\end{equation*}

We define

\begin{itemize}
\item[\refstepcounter{equation}\text{(\theequation)}\label{cs6}] $\Gamma
_{0}^{+}=\left\{ 
\begin{array}{c}
P\in \mathcal{P}^{+}:\left\vert \partial ^{\beta }P\left( 0\right)
\right\vert \leq 1\text{ for }\left\vert \beta \right\vert \leq m\text{; }%
P\left( x\right) +\left\vert x\right\vert ^{m}\geq 0\text{ for all }x\in 
\mathbb{R}^{n}\text{;} \\ 
\text{and for every }\epsilon >0\text{, there exists }\delta >0\text{ such
that} \\ 
\text{ }P\left( x\right) +\epsilon \left\vert x\right\vert ^{m}\geq 0\text{
for }\left\vert x\right\vert \leq \delta \text{.}%
\end{array}%
\right\} $.
\end{itemize}

Later, we will discuss how $\Gamma^+_0$ may be computed in principle.

We next establish the following result.

\begin{lemma}
\label{lemma-ccs} For small enough $c$ and large enough $C$, the following
hold.

\begin{description}
\item[(A)] If $F \in C^m(\mathbb{R}^n)$, $\|F\|_{C^m(\mathbb{R}^n)} \leq c$, 
$F \geq 0$ on $\mathbb{R}^n$, then $J_0^+(F) \in \Gamma_0^+$.

\item[(B)] If $P \in \Gamma_0^+$, then there exists $F \in C^m(\mathbb{R}^n)$
such that $\|F\|_{C^m(\mathbb{R}^n)} \leq C$, $F \geq 0$ on $\mathbb{R}^n$,
and $J_0^+(F) = P$.
\end{description}
\end{lemma}

\begin{proof}
(A) follows trivially from Taylor's theorem. We prove (B).

Let $P\in \Gamma _{0}^{+}$ be given. We introduce cutoff functions $\varphi $%
, $\chi \in C^{m}\left( \mathbb{R}^{n}\right) $ with the following properties.

\begin{itemize}
\item[\refstepcounter{equation}\text{(\theequation)}\label{cs7}] $\left\Vert
\chi \right\Vert _{C^{m}\left( \mathbb{R}^{n}\right) }\leq C$, $\chi =1$ in
a neighborhood of $0$, $\chi =0$ outside $B_{n}\left( 0,1/2\right) $, and $%
0\leq \chi \leq 1$ on $\mathbb{R}^{n}$.
\end{itemize}

\begin{itemize}
\item[\refstepcounter{equation}\text{(\theequation)}\label{cs8}] $\left\Vert
\varphi \right\Vert _{C^{m}\left( \mathbb{R}^{n}\right) }\leq C$, $\varphi
=1 $ for $1/2\leq \left\vert x\right\vert \leq 2$, $\varphi \geq 0$ on $\mathbb{R}%
^{n}$,
\end{itemize}
and $\varphi \left( x\right)
=0$ unless $1/4<\left\vert x\right\vert <4$.

For $k \geq 0$, let

\begin{itemize}
\item[\refstepcounter{equation}\text{(\theequation)}\label{cs9}] $\varphi
_{k}\left( x\right) =\varphi \left( 2^{k}x\right) $ $\left( x\in \mathbb{R}%
^{n}\right) $.
\end{itemize}

Thus,

\begin{itemize}
\item[\refstepcounter{equation}\text{(\theequation)}\label{cs10}] $%
\left\Vert \varphi _{k}\right\Vert _{C^{m}\left( \mathbb{R}^{n}\right) }\leq
C2^{mk}$, $\varphi _{k}\geq 0$ on $\mathbb{R}^{n}$, $\varphi _{k}\left(
x\right) =1$ for $2^{-1-k}\leq \left\vert x\right\vert \leq 2^{1-k}$, $%
\varphi _{k}\left( x\right) =0$ unless $2^{-2-k}\leq \left\vert x\right\vert
\leq 2^{2-k}$.
\end{itemize}

Also, for $k \geq 0$, we define a real number $b_k$ as follows.

\begin{itemize}
\item[\refstepcounter{equation}\text{(\theequation)}\label{cs11}] $b_{k}=0$
if $P\left( x\right) \geq 0$ for $\left\vert x\right\vert \leq 2^{-k}$; $%
b_{k}=-\min \left\{ P\left( x\right) :\left\vert x\right\vert \leq
2^{-k}\right\} $ otherwise.
\end{itemize}

Since $P \in \Gamma_0^+$, the $b_k$ satisfy the following:

\begin{itemize}
\item[\refstepcounter{equation}\text{(\theequation)}\label{cs12}] $0\leq
b_{k}\leq 2^{-mk}$ for all $k\geq 0$.
\end{itemize}

\begin{itemize}
\item[\refstepcounter{equation}\text{(\theequation)}\label{cs13}] $%
b_{k}\cdot 2^{mk}\rightarrow 0$ as $k\rightarrow \infty $.
\end{itemize}

By definition of the $b_k$, we have also for each $k \geq 0$ that

\begin{itemize}
\item[\refstepcounter{equation}\text{(\theequation)}\label{cs14}] $P\left(
x\right) +b_{k}\geq 0$ for $\left\vert x\right\vert \leq 2^{-k}$.
\end{itemize}

We define a function $\tilde{F}$ on the closed unit ball $\overline{B_n(0,1)}$ by setting

\begin{itemize}
\item[\refstepcounter{equation}\text{(\theequation)}\label{cs15}] $\tilde{F}%
\left( x\right) =P\left( x\right) +\sum_{k=0}^{\infty }b_{k}\varphi
_{k}\left( x\right) $ for $x\in \overline{B_{n}\left( 0,1\right)} $.
\end{itemize}
(The sum contains
at most $C$ nonzero terms for any given $x$.)

We will check that

\begin{itemize}
\item[\refstepcounter{equation}\text{(\theequation)}\label{cs16}] $\tilde{F}%
\geq 0$ on $\overline{B_{n}\left( 0,1\right)} $.
\end{itemize}

Indeed, $\tilde{F}\left( 0\right) =P(0)\geq 0$ since each $\varphi _{k}(0)=0$
and $P\in \Gamma _{0}^{+}$. For $\hat{x}\in \overline{B_{n}(0,1)}\setminus \{0\}$ we
have $2^{-1-\hat{k}}\leq |\hat{x}|\leq 2^{-\hat{k}}$ for some $\hat{k}\geq 0$%
.

We then have $\varphi _{\hat{k}}(\hat{x})=1$ by \eqref{cs10}, hence $P(\hat{x%
})+b_{\hat{k}}\varphi _{\hat{k}}(\hat{x})\geq 0$ by \eqref{cs14}. Since also 
$b_{k}\varphi _{k}(\hat{x})\geq 0$ for all $k$, it follows that $$\tilde{F}(%
\hat{x})=\left[ P\left( \hat{x}\right) +b_{\hat{k}}\varphi _{\hat{k}}\left( 
\hat{x}\right) \right] +\sum_{k\not=\hat{k}}b_{k}\varphi _{k}\left( x\right)
\geq 0,$$ completing the proof of \eqref{cs16}.

Next, we check that

\begin{itemize}
\item[\refstepcounter{equation}\text{(\theequation)}\label{cs17}] $\tilde{F}%
\in C^{m}\left( \overline{B_{n}\left( 0,1\right)} \right) $, $\left\Vert \tilde{F}%
\right\Vert _{C^{m}\left( \overline{B_{n}\left( 0,1\right)} \right) }\leq C$, $%
J_{0}^{+}\left( \tilde{F}\right) =P$.
\end{itemize}

To see this, let

\begin{itemize}
\item[\refstepcounter{equation}\text{(\theequation)}\label{cs18}] $\tilde{F}%
_{K}=P+\sum_{k=0}^{K}b_{k}\varphi _{k}$ for $K\geq 0$.
\end{itemize}

Since $P\in \Gamma _{0}^{+}$, we have $\left\vert \partial ^{\beta }P\left(
0\right) \right\vert \leq 1$ for $\left\vert \beta \right\vert \leq m$, hence

\begin{itemize}
\item[\refstepcounter{equation}\text{(\theequation)}\label{cs19}] $%
\left\Vert P\right\Vert _{C^{m}\left(\overline{ B_{n}\left( 0,1\right)} \right) }\leq C$%
.
\end{itemize}

Also, \eqref{cs10} and \eqref{cs12} give%
\begin{equation*}
\left\Vert b_{k}\varphi _{k}\right\Vert _{C^{m}\left(\overline{ B_n\left( 0,1\right)}
\right) }\leq C\text{ for each }k\text{.}
\end{equation*}

Since any given $x \in \overline{ B_n(0,1)}$ belongs to at most $C$ of the supports of
the $\varphi_k$, it follows that

\begin{itemize}
\item[\refstepcounter{equation}\text{(\theequation)}\label{cs20}] $%
\left\Vert \sum_{k=0}^{K}b_{k}\varphi _{k}\right\Vert _{C^{m}\left(
\overline{B_{n}\left( 0,1\right)} \right) }\leq C$.
\end{itemize}

From \eqref{cs18}, \eqref{cs19}, \eqref{cs20}, we see that

\begin{itemize}
\item[\refstepcounter{equation}\text{(\theequation)}\label{cs21}] $\tilde{F}%
_{K}\in C^{m}\left( \overline{B_{n}\left( 0,1\right)} \right) $ and $\left\Vert \tilde{F%
}\right\Vert _{C^{m}\left( \overline{B_{n}\left( 0,1\right)} \right) }\leq C$.
\end{itemize}

Also, \eqref{cs10} and \eqref{cs18} tell us that

\begin{itemize}
\item[\refstepcounter{equation}\text{(\theequation)}\label{cs22}] $%
J_{0}^{+}\left( \tilde{F}_{K}\right) =P$ for each $K$.
\end{itemize}

Furthermore for $K_{1} < K_{2}$, \eqref{cs18} gives $\tilde{F}_{K_{2}}-%
\tilde{F}_{K_{1}}=\sum_{K_{1}<k\leq K_{2}}b_{k}\varphi _{k}$. Let $\epsilon
>0$. From \eqref{cs10} and \eqref{cs13} we see that 
\begin{equation*}
\max_{K_{1}<k\leq K_{2}}\left\Vert b_{k}\varphi _{k}\right\Vert
_{C^{m}\left( \overline{B_{n}\left( 0,1\right)} \right) }<\epsilon \text{ if }K_{1}%
\text{ is large enough.}
\end{equation*}%
Since any given point lies in support $\varphi _{k}$ for at most $C$
distinct $k$, it follows that%
\begin{equation*}
\left\Vert \sum_{K_{1}<k\leq K_{2}}b_{k}\varphi _{k}\right\Vert
_{C^{m}\left( \overline{B_{n}\left( 0,1\right)} \right) }\leq C\epsilon \text{ if }%
K_{2}>K_{1}\text{ and }K_{1}\text{ is large enough.}
\end{equation*}

Thus, $(\tilde{F}_{K})_{K\geq 0}$ is a Cauchy sequence in $C^{m}(\overline{B_{n}(0,1)})$%
. Consequently, $\tilde{F}_{K}\rightarrow \tilde{F}_{\infty }$ in $%
C^{m}(\overline{B_{n}(0,1)})$-norm for some $\tilde{F}_{\infty }\in C^{m}(\overline{B_{n}(0,1))}$.
From \eqref{cs21} and \eqref{cs22}, we have%
\begin{equation*}
\left\Vert \tilde{F}_{\infty }\right\Vert _{C^{m}\left( \overline{B_{n}\left(
0,1\right)} \right) }\leq C\text{ and }J_{0}^{+}\left( \tilde{F}_{\infty
}\right) =P\text{.}
\end{equation*}

On the other hand, comparing \eqref{cs15} to \eqref{cs18}, and recalling
that any given $x$ belongs to support $\theta_k$ for at most $C$ distinct $k$%
, we conclude that $\tilde{F}_K \rightarrow \tilde{F}$ pointwise as $K
\rightarrow \infty$.

Since also $\tilde{F}_K \rightarrow \tilde{F}_\infty$ pointwise as $K
\rightarrow \infty$, we have $\tilde{F}_\infty = \tilde{F}$.

Thus, $\tilde{F}\in C^{m}\left( \overline{B_{n}\left( 0,1\right)} \right) $, $%
\left\Vert \tilde{F}\right\Vert _{C^{m}\left(\overline{ B_{n}\left( 0,1\right)} \right)
}\leq C$, and $J_{0}^{+}\left( \tilde{F}\right) =P$, completing the proof of %
\eqref{cs17}.

Finally, we recall the cutoff function $\chi$ from \eqref{cs7}, and define $%
F=\chi \tilde{F}$ on $\mathbb{R}^n$.

From \eqref{cs16}, \eqref{cs17}, and the properties \eqref{cs7} of $\chi $,
we conclude that $F\in C^{m}\left( \mathbb{R}^{n}\right) $, $\left\Vert
F\right\Vert _{C^{m}\left( \mathbb{R}^{n}\right) }\leq C$, $F\geq 0$ on $%
\mathbb{R}^{n}$, and $J_{0}^{+}\left( F\right) =P$.

Thus, we have established (B).

The proof of Lemma \ref{lemma-ccs} is complete.
\end{proof}

Now let $\pi :\mathcal{P}^{+}\rightarrow \mathcal{P}$ denote the natural
projection from $m$-jets at $0$ to $\left( m-1\right) $-jets at $0$, namely, 
$\pi P=J_{0}\left( P\right) $ for $P\in \mathcal{P}^{+}$.

We then set $\tilde{\Gamma}_{0}=\pi \Gamma _{0}^{+}$.

From the above lemma, we learn the following.

\begin{description}
\item[$\left( A^{\prime }\right) $] Let $F\in C^{m}\left( \mathbb{R}%
^{n}\right) $ with $\left\Vert F\right\Vert _{C^{m}\left( \mathbb{R}%
^{n}\right) }\leq c$, $F\geq 0$ on $\mathbb{R}^{n}$. Then $J_{0}\left(
F\right) \in \tilde{\Gamma}_{0}$.

\item[$\left( B^{\prime }\right) $] Let $P\in \tilde{\Gamma}_{0}$. Then
there exists $F\in C^{m}\left( \mathbb{R}^{n}\right) $ such that $\left\Vert
F\right\Vert _{C^{m}\left( \mathbb{R}^{n}\right) }\leq C$, $F\geq 0$ on $%
\mathbb{R}^{n}$, and $J_{0}\left( F\right) =P$.
\end{description}

Recalling the definition (\ref{cs1}), we conclude from ($A^{\prime }$), ($%
B^{\prime }$) that $\Gamma _{\ast }\left( 0,c\right) \subset \tilde{\Gamma}%
_{0}\subset \Gamma _{\ast }\left( 0,C\right) $.

Thus, our $\tilde{\Gamma}_{0}$ satisfies the key condition (\ref{cs5}).

We discuss briefly how the convex set $\tilde{\Gamma}_{0}$ may be computed
in principle. Recall \cite{hormander-2007} that a semialgebraic set is a
subset of a vector space obtained by taking finitely many unions,
intersections, and complements of sets of the form $\left\{ P>0\right\} $
for polynomials $P$. Any subset of a vector space $V$ defined by $E=\left\{
x\in V:\Phi \left( x\right) \text{ is true}\right\} $, where $\Phi $ is a
formula of first-order predicate calculus (for the theory of real-closed
fields) is semialgebraic; moreover, there is an algorithm that accepts $\Phi 
$ as input and exhibits $E$ as a Boolean combination of sets of the form $%
\left\{ P>0\right\} $ for polynomials $P$. For any given $m$, $n$, we see,
by inspection of the definitions of $\Gamma _{0}^{+}$ and $\tilde{\Gamma}_{0}
$, that $\Gamma _{0}^{+}\subset \mathcal{P}^{+}$ is defined by a formula of
first-order predicate calculus; hence, the same holds for $\tilde{\Gamma}%
_{0}\subset \mathcal{P}$.

Therefore, in principle, we can compute $\tilde{\Gamma}_{0}$ as a Boolean
combination of sets of the form $\left\{ P\in \mathcal{P}:\Pi \left(
P\right) >0\right\} $, where $\Pi $ is a polynomial on $\mathcal{P}$.

In practice, we make no claim that we know how to compute $\tilde{\Gamma}%
_{0} $.

It would be interesting to give a more practical method to compute a convex
set satisfying (\ref{cs5}).

\section{$C^{m-1,1}$ Interpolation by Nonnegative Functions}

\label{afc}

In this section we will establish Theorem \ref{Th4} (B) and discuss computational issues for $C^{m-1,1}$ interpolation by nonnegative functions. 

We note that the derivatives $\partial ^{\beta }F$ of $F\in C^{m-1,1}\left( 
\mathbb{R}^{n}\right) $ of order $\left\vert \beta \right\vert \leq m-1$ are
continuous. Also, Taylor's theorem holds in the form 
\begin{equation*}
\left\vert \partial ^{\beta }F\left( y\right) -\sum_{\left\vert \beta
\right\vert +\left\vert \gamma \right\vert \leq m-1}\frac{1}{\gamma !}\left[
\partial ^{\gamma +\beta }F\left( x\right) \right] \cdot \left( y-x\right)
^{\gamma }\right\vert \leq C\left\Vert F\right\Vert _{C^{m-1,1}\left( 
\mathbb{R}^{n}\right) }\cdot \left\vert y-x\right\vert ^{m-\left\vert \beta
\right\vert }
\end{equation*}%
for $x,y\in \mathbb{R}^{n}$.

Similar remarks apply to $C^{m-1,1}\left( Q\right) $ and $C^{m}\left(
Q\right) $ for cubes $Q\subset \mathbb{R}^{n}$.

Therefore, we may repeat the proofs \cite{fil-2016} of Lemmas \ref{lemma-inf1} and \ref%
{lemma-inf2} in Section \ref{inf}, to derive the following results.

\begin{lemma}
\label{lemma-inf1'}For $x\in \mathbb{R}^{n}$, $M>0$, let 
\begin{equation*}
\Gamma _{\ast }^{\prime }\left( x,M\right) =\left\{ 
\begin{array}{c}
P\in \mathcal{P}:\exists F\in C^{m-1,1}\left( \mathbb{R}^{n}\right) \text{
such that } \\ 
\left\Vert F\right\Vert _{C^{m-1,1}\left( \mathbb{R}^{n}\right) }\leq
M,F\geq 0\text{ on }\mathbb{R}^{n}\text{, }J_{x}\left( F\right) =P%
\end{array}%
\right\} \text{.}
\end{equation*}%
Let $f:E\rightarrow \lbrack 0,\infty )$, where $E\subset \mathbb{R}^{n}$ is
finite. For $x\in E$, $M>0$, let 
\begin{equation*}
\Gamma _{f}^{\prime }\left( x,M\right) =\left\{ P\in \Gamma _{\ast }^{\prime
}\left( x,M\right) :P\left( x\right) =f\left( x\right) \right\} \text{.}
\end{equation*}%
Then $\vec{\Gamma}_{f}^{\prime }:=\left( \Gamma _{f}^{\prime }\left(
x,M\right) \right) _{x\in E,M>0}$ is a $\left( C,1\right) $-convex shape
field, where $C$ depends only on $m$, $n$.
\end{lemma}

\begin{lemma}
\label{lemma-inf2'}Let $E$, $f$, $\Gamma _{\ast }^{\prime }\left( x,M\right) 
$ be as in Lemma \ref{lemma-inf1'}, and let $M>0$, $\vec{P}=\left(
P^{x}\right) _{x\in E}\in Wh\left( E\right) $. Suppose we have $P^{x}\in
\Gamma _{\ast }^{\prime }\left( x,M\right) $ for all $x \in E$, and $%
\left\vert \partial ^{\beta }\left( P^{x}-P^{y}\right) \left( x\right)
\right\vert \leq M\left\vert x-y\right\vert ^{m-\left\vert \beta \right\vert
}$ for $x,y\in E$, $\left\vert \beta \right\vert \leq m-1$. Then there
exists $F\in C^{m-1,1}\left( \mathbb{R}^{n}\right) $ such that $J_{x}\left(
F\right) =P^{x}$ for all $x\in E$, and $\left\Vert F\right\Vert
_{C^{m-1,1}\left( \mathbb{R}^{n}\right) }\leq CM$, where $C$ depends only on 
$m$, $n$.
\end{lemma}

Similarly, by making small changes in the proof \cite{fil-2016} of Theorem \ref%
{theorem-fp-for-nonnegative-interpolation}, we obtain the following result.

\begin{lemma}
\label{lemma-inf3'}There exist $k^{\#}$, $C$, depending only on $m$, $n$ for
which the following holds.

Let $E\subset \mathbb{R}^{n}$ be finite, let $f:E\rightarrow \lbrack
0,\infty )$, and let $M_{0}>0$. Suppose that for each $S\subset E$ with $%
\#\left( S\right) \leq k^{\#}$ there exists $\vec{P}^{S}=\left( P^{x}\right)
_{x\in S}\in Wh\left( S\right) $ such that $P^{x}\in \Gamma _{f}^{\prime
}\left( x,M_{0}\right) $ for all $x\in S$, and $\left\vert \partial ^{\beta
}\left( P^{x}-P^{y}\right) \right\vert \leq M_{0}\left\vert x-y\right\vert
^{m-\left\vert \beta \right\vert }$ for $x,y\in S$, $\left\vert \beta
\right\vert \leq m-1$.

Then there exists $F\in C^{m-1,1}\left( \mathbb{R}^{n}\right) $ such that $%
\left\Vert F\right\Vert _{C^{m-1,1}\left( \mathbb{R}^{n}\right) }\leq CM_{0}$%
, $F\geq 0$ on $\mathbb{R}^{n}$, and $F=f$ on $E$.
\end{lemma}


Now we can easily deduce the following result.

\begin{theorem}[Finiteness Principle for Nonnegative $C^{m-1,1}$%
-Interpolation]
\label{theorem-fp-for-cm-1-interpolation}There exists constants $k^{\#}$, $C$%
, depending only on $m,n$ for which the following holds.

Let $f:E\rightarrow \lbrack 0,\infty )$, with $E \subset \mathbb{R}^n$
arbitrary (not necessarily finite). Let $M_0 >0$. Suppose that for each $S
\subset E$ with $\#(S) \leq k^\#$ there exists $\vec{P}=(P^x)_{x\in S}\in
Wh(S)$ such that

\begin{itemize}
\item $P^x \in \Gamma_f^{\prime }(x,M_0)$ for all $x \in S$,

\item $\left\vert \partial ^{\beta }\left( P^{x}-P^{y}\right) \left(
x\right) \right\vert \leq M_{0}\left\vert x-y\right\vert ^{m-\left\vert
\beta \right\vert }$ for $x,y\in S$, $\left\vert \beta \right\vert \leq m-1$.
\end{itemize}

Then there exists $F\in C^{m-1,1}\left( \mathbb{R}^{n}\right) $ such that

\begin{itemize}
\item $||F||_{C^{m-1,1}(\mathbb{R}^n)}\leq CM_0$,

\item $F \geq 0$, and

\item $F=f$ on $E$.
\end{itemize}
\end{theorem}

\begin{proof}
Suppose first that $E\subset Q$ for some cube $Q\subset \mathbb{R}^{n}$.
Then by Ascoli's theorem,%
\begin{equation*}
\left\{ F\in C^{m-1,1}\left( Q\right) :\left\Vert F\right\Vert
_{C^{m-1,1}\left( Q\right) }\leq CM_{0}\text{, }F\geq 0\text{ on }Q\right\}
\equiv X
\end{equation*}
is compact in the $C^{m-1}(Q)$-norm topology.

For each finite $E_0 \subset E$, Lemma \ref{lemma-inf3'} tells us that there
exists $F \in X$ such that $F = f$ on $E_0$.

Consequently, there exists $F \in X$ such that $F = f$ on $E$. That is,

\begin{itemize}
\item[\refstepcounter{equation}\text{(\theequation)}\label{afc1}] $F\in
C^{m-1,1}\left( Q\right) $, $\left\Vert F\right\Vert _{C^{m-1,1}\left(
Q\right) }\leq CM_{0}$, $F\geq 0$ on $Q$, $F=f$ on $E$.
\end{itemize}

We have achieved \eqref{afc1}, assuming that $E \subset Q$.

Now suppose $E \subset \mathbb{R}^n$ is arbitrary.

We introduce a partition of unity $1=\sum_{\nu }\theta _{\nu }$ on $\mathbb{R%
}^{n}$, with $\theta _{\nu }\geq 0$ on $\mathbb{R}^{n}$, $\theta _{\nu }\in
C^{m}\left( \mathbb{R}^{n}\right) $, $\left\Vert \theta _{\nu }\right\Vert
_{C^{m}\left( \mathbb{R}^{n}\right) }\leq C$, support $\theta _{\nu }\subset
Q_{\nu }$ for a cube $Q_{\nu }\subset \mathbb{R}^{n}$, with (say) $\delta
_{Q_{\nu }}=1$, and such that any given $x\in \mathbb{R}^{n}$ has a
neighborhood that intersects at most $C$ of the $Q_{\nu }$. (Here $C$
depends only on $m,n$.)

Applying our result \eqref{afc1} to $f|_{E\cap Q_{\nu }}:E\cap Q_{\nu
}\rightarrow \lbrack 0,\infty )$ for each $\nu $, we obtain functions $%
F_{\nu }\in C^{m-1,1}\left( Q_{\nu }\right) $ such that $\left\Vert F_{\nu
}\right\Vert _{C^{m-1,1}\left( Q_{\nu }\right) }\leq CM_{0}$, $F_{\nu }\geq
0 $ on $Q_{\nu }$, $F_{\nu }=f$ on $E\cap Q_{\nu }$.

(Here $C$ depends only on $m, n$.)

We define $F=\sum_{\nu }\theta _{\nu }F_{\nu }$ on $\mathbb{R}^{n}$. One
checks easily that $\left\Vert F\right\Vert _{C^{m-1,1}\left( \mathbb{R}%
^{n}\right) }\leq C^{\prime }M_{0}$ with $C^{\prime }$ determined by $m$, $n$%
; $F\geq 0$ on $\mathbb{R}^{n}$; and $F=f$ on $E$.

This completes the proof of Theorem \ref{theorem-fp-for-cm-1-interpolation}.
\end{proof}

Note that Theorem \ref{theorem-fp-for-cm-1-interpolation} easily implies Theorem \ref{Th4} (B).

As in the case of nonnegative $C^{m}$-interpolation, we want to replace $%
\Gamma _{f}^{\prime }(x,M)$ by something easier to calculate. In the $%
C^{m-1,1}$-setting, it is enough to make the following observation.

Define \[
\tilde{\Gamma}_{0}^{\prime }=\left\{ 
\begin{array}{c}
P\in \mathcal{P}:\left\vert \partial ^{\beta }P\left( 0\right) \right\vert
\leq 1\text{ for }\left\vert \beta \right\vert \leq m-1\text{ and} \\ 
\text{ }P\left( x\right) +\left\vert x\right\vert ^{m}\geq 0\text{ for all }%
x\in \mathbb{R}^{n}%
\end{array}%
\right\} \text{.}
\]

Then

\begin{itemize}
\item[\refstepcounter{equation}\text{(\theequation)}\label{afc2}] $\Gamma
_{\ast }^{\prime }\left( 0,c\right) \subset \tilde{\Gamma}_{0}^{\prime
}\subset \tilde{\Gamma}_{\ast }^{\prime }\left( 0,C\right) $ with $c$, $C$
depending only on $m$, $n$.
\end{itemize}

Indeed, the first inclusion in \eqref{afc2} is immediate from the
definitions and Taylor's theorem. To prove the second inclusion, we let $P
\in \tilde{\Gamma}^{\prime }_0$ be given, and set $F(x)=\chi(x)(P(x)+|x|^m)$%
, where $\chi$ is a nonnegative $C^m$ function with norm at most $C_{\ast}$
(depending only on $m$, $n$), satisfying $J_0(\chi)=1$ and support $\chi
\subset B_n(0,1)$.

We then have $F\in C^{m-1,1}(\mathbb{R}^{n})$, $\left\Vert F\right\Vert
_{C^{m-1,1}\left( \mathbb{R}^{n}\right) }\leq C$ (depending only on $m$, $n$%
), $F\geq 0$ on $\mathbb{R}^{n}$, $J_{0}\left( F\right) =P$. Hence, $P\in
\Gamma _{\ast }^{\prime }\left( 0,C\right) $, completing the proof of %
\eqref{afc2}.

This concludes our discussion of interpolation by nonnegative $C^{m-1,1}$
functions.

\bigskip

\bibliographystyle{plain}
\def\cprime{$'$} \def\cprime{$'$}

\vspace{1mm} 
\newpage

\begin{flushleft} 
Charles Fefferman \\
Affiliation: Department of Mathematics, Princeton University, Fine Hall Washington Road, Princeton, New Jersey, 08544, USA \\
Email: cf$\MVAt$math.princeton.edu  \\
\vspace{2mm} 
Arie Israel \\ 
Affiliation: The University of Texas at Austin, 
Department of Mathematics,
2515 Speedway Stop C1200, 
Austin, Texas, 78712-1202, USA \\ 
Email: arie$\MVAt$math.utexas.edu \\
\vspace{2mm} 
Garving K. Luli \\ 
Affiliation: Department of Mathematics, 
University of California, Davis, 
One Shields Ave, 
Davis, California, 95616, USA \\ 
Email: kluli@$\MVAt$math.ucdavis.edu 
\end{flushleft}

\end{document}